\documentclass[12pt]{article}

\usepackage[utf8x]{inputenc}
\usepackage[english]{babel}
\usepackage{amssymb, amsmath, amsthm, verbatim, cite, lscape, longtable, graphicx, wrapfig, tikz}

\usepackage{url}

\newtheorem{theorem}{Theorem}
\newtheorem{lemma}[theorem]{Lemma}
\newtheorem{corollary}[theorem]{Corollary}

\usepackage{graphicx}

\begin{document}

\title{The Diameter of Sum Basic Equilibria Games}

\author{Aida Abiad\thanks{\texttt{a.abiad.monge@tue.nl}, Department of Mathematics and Computer Science, Eindhoven University of Technology, The Netherlands}
\thanks{Department of Mathematics: Analysis, Logic and Discrete Mathematics, Ghent University, Belgium}\thanks{Department of Mathematics and Data Science of Vrije Universiteit Brussel, Belgium} \qquad Carme Alvarez \thanks{\texttt{alvarez@cs.upc.edu},  Computer Science Department, Polytechnic University of Catalonia, Spain} \qquad Arnau Messegu\'e\thanks{\texttt{arnau.messegue@upc.edu},  Computer Science Department, Polytechnic University of Catalonia, Spain}}

\date{}
 
\maketitle   




\begin{abstract}
A graph $G$ of order $n$ is said to be a \emph{sum basic equilibrium} if and only if for every edge $uv$ from $G$ and any node $v'$ from $G$, when performing the swap of the edge $uv$ for the edge $uv'$, the sum of the distances from $u$ to all the other nodes is not strictly reduced. This concept lies in the heart of the so-called network creation games, where the central problem is to understand the structure of the resulting equilibrium graphs, and in particular, how well they globally minimize the diameter. It was shown in [Alon, Demaine,  Hajiaghayi, Leighton, \emph{SIAM J. Discrete Math.} 27(2), 2013] that the diameter of sum basic equilibria is $2^{O(\sqrt{\log n})}$ in general, and at most $2$ for trees. In this paper we show that the upper bound of $2$ can be extended to bipartite graphs, and that it also holds for some nonbipartite classes like block graphs and cactus graphs. 

\end{abstract}



\section{Introduction}
\textbf{Definition of the model and context.} First of all, let us introduce some terminology that will help us to clarify the notion of a \emph{sum basic equilibrium network}. Let $G$ be a connected and undirected graph of size $n$, and let $u$ be a node from $G$. A \emph{deviation} in $u$ is any swap of an edge $uv$ from $G$ for any other edge $uv'$ with $v' \neq u,v$ any other node from $G$. The \emph{deviated graph} associated to any  such deviation is the resulting graph obtained after applying the swap. Furthermore, the \emph{cost difference} associated to any deviation in $u$ is  the difference between the sum of distances from $u$ to all the other nodes in the deviated graph minus the sum of the distances from $u$ to all the other nodes in the original graph. A connected and undirected graph $G$ is a \emph{sum basic equilibrium} iff for every node $u$ in $G$ the cost difference associated to every possible deviation in $u$  is non-negative. 

This notion of a sum basic equilibrium was first introduced by Alon et al. in 2010 \cite{AlonDHL:10}
 and is inspired by the \emph{sum classical network creation game}, which was introduced by Fabrikant et al. in 2003 \cite{Fe:03}. In this model, the sum classical network creation game, two parameters are considered: $n$ the size of the network and $\alpha$ the price of buying any single link. One of the main interests in the sum classical network creation game is to study the price of anarchy, a measure that quantifies the loss of efficiency of a system due to the selfish behavior of its agents. The price of anarchy, provides a quantitative understanding of the behaviour of such Internet-like networks and, interestingly, up until now it has been shown that for almost any $\alpha$ the price of anarchy is asymptotically constant \cite{Demaineetal:07, Alvarezetal3}. One of the results used to prove this is that the price of anarchy is upper bounded by the diameter of equilibrium networks plus one unit \cite{Demaineetal:07}. Hence, proving that the diameter of equilibria is small allows us to deduce that the price of anarchy is small too. This is one of the reasons that explain why we are interested in finding the best possible non-trivial upper bounds for the diameter of equilibrium networks. 
 
In this regard, one of the most important contributions by Alon et al. \cite{AlonDHL:10} is a general upper bound on the diameter of any sum basic equilibrium of $2^{O(\sqrt{ \log n})}$. However, this bound can be dramatically reduced if we restrict to the tree topology, in which case the diameter is shown to be at most $2$:

\vskip 5pt
\textbf{Theorem 1 \cite{AlonDHL:10}}.  \label{thm:basic} \emph{If a sum equilibrium graph in the basic
network-creation game is a tree, then it has diameter at
most 2, and thus is a star}.
\vskip 5pt

Moreover, in \cite{AlonDHL:10}, the authors establish a connection between sum basic equilibria of diameter larger than $2\log n$ and \emph{distance-uniform graphs}. While the authors conjecture that distance-uniform graphs have logarithmic diameter, which would imply poly-logarithmic diameter for sum basic equilibria, Lavrov et al. \cite{LaLo:17} later refute this conjecture. Later, Nikoletseas et al. \cite{Structurebasic} use the probability principle to establish structural properties of sum basic equilibria. From some of these properties, it is shown that in some extremal situations, like when the maximum degree of the equilibrium network is at least $n/\log^l n$ with $l > 0$, 
the diameter is polylogarithmic.

\vskip 5pt

\textbf{Our Contribution.} In this work we extend Theorem \ref{thm:basic} to bipartite graphs, block graphs and cactus graphs, proving that the diameter of any such sum basic equilibrium networks is at most 2. 
This paper is structured as follows. In Section \ref{sec:localswap} we investigate the bipartite case. For this, our approach consists in considering any $2-$edge-connected component $H$ of a non-tree sum basic equilibrium $G$. We first consider all the collection of individual swaps $uv$ for $uv'$ for each $u,v,v' \in V(H)$ and $uv,vv' \in E(H)$. We show that if $diam(H )> 2$, then the sum of the cost differences of all these swaps will be $<0$, thus contradicting the fact that $G$ is a sum basic equilibrium.  In the end of this section, we study further elementary properties of any $2-$edge-connected component of any non-tree sum basic equilibrium that work in general which allow us to reach the conclusion.
In Section \ref{sec:other} we consider the non-bipartite case. In particular,  we analyse further properties that sum basic equilibrium block graphs and sum basic equilibrium cactus graphs satisfy, imposing that appropriate single-edge swaps cannot improve the sum of distances to all the nodes from the network if the diameter of such networks is greater than $2$.

\vskip 5pt

\textbf{Notation.} In this work we consider mainly undirected graphs $G$ for which we denote by $V(G),E(G)$ its corresponding sets of vertices and edges, respectively. 

Given an undirected graph $G$ and any pair of nodes $u,v$ from $G$ we denote by $d_G(u,v)$ the distance between $u,v$. In this way, $D(u)$ is the sum of distances from $u$ to all the other nodes, that is, $D(u)=\sum_{v \neq u}d_G(u,v)$ if $G$ is connected or $\infty$ otherwise. Moreover, if $Z$ is a subgraph of $G$, noted as $Z \subseteq G$, then we write $D_Z(u) = \sum_{v \in V(Z)}d_G(u,v)$.

Now let $H$ be a subgraph from $G$. The $i-$th distance layer in $H$ with respect $u$ is denoted as $\Gamma_{i,H}(u)=\left\{v \in V(H) \mid d_G(u,v)=i \right\}$. In particular, the neighbourhood of $u$ in $H$, the set of nodes from $V(H)$ at distance one with respect $u$, is $\Gamma_{1,H}(u)$. Furthermore, if $P$ is a property, we say that $H$ is a maximal subgraph of $G$ satisfying $P$ when for any other subgraph $H'$ of $G$, if $H'$ satisfies $P$ then $H \not \subseteq H'$. 

Then, an edge $e \in E(G)$ is said to be a \emph{bridge} if its removal increases the number of connected components from $G$ and a $2-$edge-connected component $H$ from $G$ is any maximal connected subgraph of $G$ not containing any bridge. Similarly, a node $v \in V(G)$ is a cut vertex iff its removal increases the number of connected components of $G$ and then a biconnected component $H$ from $G$ is any maximal connected subgraph of $G$ not containing any cut vertex. Finally, for a given $2$-edge-connected or biconnected component $H$ from $G$ and a vertex $u\in V(H)$, $W_H(u)$ is the connected component containing $u$ in the subgraph induced by the vertices $(V(G) \setminus V(H)) \cup \left\{u \right\}$.  


\section{Bipartite graphs}
\label{sec:localswap}

Given a non-tree bipartite sum basic equilibrium graph $G$, let $H$ be any of its $2-$edge-connected components. We first show that $diam(H)=2$. 

Given $u \in V(H)$ and $w\in V(G)$, we define $\delta^-_w(u)$ the subset of nodes $v$ from $\Gamma_{1,H}(u)$ such that $d_G(w,v) = d_G(w,u)-1$ and  $\delta^+_w(u)$ the subset of nodes $v$ from $\Gamma_{1,H}(u)$ such that $d_G(w,v) = d_G(w,u)+1$. Since $G$ is bipartite, for any $u\in V(H)$ and $w \in V(G)$,  $\delta_w^-(u) \cup \delta_w^+(u) = \Gamma_{1,H}(u)$.


Moreover, given $u\in V(H)$ and $w\in V(G)$ such that $|\delta^-_w(u)| = 1$, we define $u_{w}^- \in \delta_w^-(u)$ to be the neighbour of $u$ in $H$ closer from $w$ than $u$. Recall that, for any $u\in V(H)$ and $w \in V(G)$, if $|\delta^-_w(u)| = 1$ then clearly $\delta_w^+(u) \neq \emptyset$ because $H$ is $2-$edge-connected.



Now, let $u,v$ be nodes with $u\in V(H)$ and $v\in \Gamma_{1,H}(u)$. We define $S(u,v)$ to be the sum of the cost differences associated to the swaps of the edge $uv$ by the edges $uv'$ with $v' \in \Gamma_{1,H}(v) \setminus \left\{ u \right\} $ divided over $deg_H(v)-1$. Then we define $S = \sum_{u \in V(H)} \sum_{v \in \Gamma_{1,H}(u)}S(u,v)$. 

Let $u,v,w$ be nodes with $u\in V(H)$, $v\in \Gamma_{1,H}(u)$ and $w\in V(G)$ and define $\Delta_w(u,v)$ to be the sum of the distance changes from $u$ to $w$ due to the swaps of the edge $uv$ by the edges $uv'$ with $v' \in \Gamma_{1,H}(v) \setminus \left\{u\right\}$ divided over $deg_H(v)-1$. Furthermore, let $\Delta_w = \sum_{u \in V(H)} \sum_{v \in \Gamma_{1,H}(u)}\Delta_w(u,v)$. 

In this way we have that $S = \sum_{w \in V(G)} \Delta_w$. 

We first find a formula to compute the value of $\Delta_w(u,v)$ allowing us to obtain an expression for $\Delta_w$.



\begin{lemma} \label{lem:basic}
For any nodes $u,v \in V(H)$ and $w \in V(G)$ such that $v\in \Gamma_{1,H}(u)$, it holds 
\begin{equation*}
\Delta_w(u,v) = \begin{cases}
\frac{-|\delta_w^-(u_w^-)|+|\delta^+_w(u_w^-)|-1}{deg_H(v)-1} &\text{ if  $|\delta_w^-(u)|=1$ and $v=u_w^-$,}\\
 \frac{-|\delta_w^-(v)| }{deg_H(v)-1}  &\text{ if  $|\delta_w^-(u)|> 1$  and  $v\in \delta_w^-(u)$,} \\
 0 & \text{otherwise.} 
\end{cases}
\end{equation*}

\end{lemma}

\begin{proof} If $v$ is further from $w$ than $u$, then clearly $\Delta_w(u,v) = 0$. Therefore, since $G$ is bipartite the remaining case is that $v$ is closer from $w$ than $u$. We can see clearly that we need to distinguish the cases $|\delta_w^-(u)| = 1$ with $v = u_w^-$ and the case $|\delta_w^-(u)| > 1$ with $v \in \delta_w^-(u)$. In the first case the corresponding sum of distance changes from $u$ to $w$ could get positive when the set of nodes $\delta_w^+(u_w^-)$ has size at least two. In contrast, in the second case the sum of distance changes is always no greater than zero because having at least another node distinct than $v$ in the subset $\delta_w^-(u)$ guarantees that when making the corresponding deviation the distance from $u$ to $w$ does not increase. 
\end{proof} 

\begin{theorem}
\label{thm:diamH} $diam(H) \leq 2$. 
\end{theorem}

\begin{proof}


First, we claim that for every $w\in V(G)$, $\Delta_w \leq 0$. Applying Lemma \ref{lem:basic} we obtain

{\footnotesize{
\begin{align*}
\Delta_w &= \sum_{u \in V(H)} \sum_{v \in \Gamma_{1,H}(u)}  \Delta_w(u,v)  \\
&=   \left( \sum_{ \left\{ u \in V(H)   \land |\delta_w^-(u)|=1 \right\} }\frac{-|\delta_w^-(u_w^-)|+|\delta_w^+(u_w^-)|-1}{deg_H(u_w^-)-1} +\sum_{ \left\{u \in V(H) \land  |\delta_w^-(u)|>1 \right\} } \sum_{ v \in \delta_w^-(u)} \frac{-|\delta_w^-(v)|}{deg_H(v)-1} \right)  \\
&= \left( \sum_{ \left\{ u \in V(H)   \land |\delta_w^-(u)|=1 \right\} }\frac{ |\delta_w^+(u_w^-)|-1}{deg_H(u_w^-)-1} +\sum_{u \in V(H)} \sum_{ v \in \delta_w^-(u)} \frac{-|\delta_w^-(v)|}{deg_H(v)-1} \right). 
\end{align*}
}}

On the one hand:

$$ \sum_{u \in V(H)} \sum_{ v \in \delta_w^-(u)} \frac{|\delta_w^-(v)|}{deg_H(v)-1} = \sum_{v \in V(H)} \sum_{ u \in \delta_w^+(v)}   \frac{| \delta_w^-(v)| }{deg_H(v)-1}=\sum_{v \in V(H)}  \frac{ |\delta_w^-(v)| |\delta_w^+(v)|}{deg_H(v)-1}.$$

Now, let $Z_w$ be the subset of nodes $z$ from $V(H)$ such that $\delta_w^-(z) \neq \emptyset$ and $\delta_w^+(z) \neq \emptyset$. If $z \in Z_w$ then clearly $|\delta_w^-(z)| | \delta_w^+(z)|  \geq deg_H(z)-1$. One possible way to see this is the following. Since $H$ is bipartite, then $|\delta_w^-(z)|$ and $| \delta_w^+(z)|$ are positive integers that add up to $deg_H(z)$. Furthermore, any concave function defined on a closed interval attains its minimum in one of its extremes. Therefore, the conclusion follows when combining these two facts to the function $f(x) = x(deg_H(z)-x)$ defined in $[1,deg_H(z)-1]$. In this way: 

\begin{equation}\label{eq:ineq1}
\sum_{u \in V(H)} \sum_{ v \in \delta_w^-(u)} \frac{|\delta_w^-(v)|}{deg_H(v)-1} = \sum_{v \in Z_w}  \frac{ |\delta_w^-(v)| |\delta_w^+(v)|}{deg_H(v)-1}\geq \sum_{v \in Z_w} 1 =  |Z_w|. 
\end{equation}

On the other hand, for any $u$ such that $| \delta_w^-(u)|= 1$ it holds

\begin{equation}
\label{eq:ineq2}
\frac{ |\delta_w^+(u_w^-)|-1}{deg_H(u_w^-)-1}  \leq 1.
\end{equation}

Notice that the equality in (\ref{eq:ineq2}) holds exactly when $\delta_w^-(u_w^-) = \emptyset$. For any $w \in V(G)$, there exists exactly one node $t_w\in V(H)$ verifying $\delta_w^-(t_w)= \emptyset$, which is the unique node from $V(H)$ such that $w \in W_H(t_w)$. Therefore, equality in (\ref{eq:ineq2}) holds exactly for the nodes from $\Gamma_{1,H}(t_w)$. 

In this way:

\begin{equation}\label{eq:ineq3}
\sum_{ \left\{ u \in V(H)   \land |\delta_w^-(u)|=1 \right\} }\frac{ |\delta_w^+(u_w^-)|-1}{deg_H(u_w^-)-1} \leq |\left\{ u \in V(H)   \land |\delta_w^-(u)|=1 \right\}|.
\end{equation}

Notice that since $H$ is bipartite, $\Gamma_{1,H}(t_w) \subseteq \left\{ u \in V(H) \mid |\delta_w^-(u)|= 1\right\}$. Therefore, equality in (\ref{eq:ineq3}) holds only when $\Gamma_{1,H}(t_w) = \left\{ u \in V(H) \mid |\delta_w^-(u)|= 1\right\} $, otherwise, the inequality in (\ref{eq:ineq3}) is strict.

Now, recall that $\left\{ u \in V(H)   \land |\delta_w^-(u)|=1 \right\} \subseteq Z_w $ because $H$ is $2-$edge-connected. Therefore, combining $(1)$ with $(3)$: 

$$\Delta_w \leq -|Z_w| +| \left\{ u \in V(H)   \mid |\delta_w^-(u)|=1 \right\}| \leq 0 $$

as we wished to prove.

Now, suppose that  $diam(H)>2$ and take any path $\pi = x_1-x_2-x_3-\cdots$ of length $diam(H)$ inside $H$.  Then,  pick $x \in W_H(x_1)$ any node  inside $W_H(x_1)$.  Setting $w=x$ we have that $x_1= t_w$ and $x_3 \in Z_w$ but $x_3 \not \in \Gamma_{1,H}(t_w)$. If $x_3 \not \in \left\{ u \in V(H)   \mid |\delta_w^-(u)|=1 \right\}$ then the inclusion $ \left\{ u \in V(H) \mid |\delta_w^-(u)|=1 \right\} \subseteq Z_w$ is strict and then $\Delta_w < 0$. Otherwise, $x_3 \in  \left\{ u \in V(H) \mid |\delta_w^-(u)|=1 \right\}$ but $x_3 \not \in \Gamma_{1,H}(t_w)$ so that the inclusion $\Gamma_{1,H}(t_w) \subseteq  \left\{ u \in V(H) \mid |\delta_w^-(u)|=1 \right\}$ is strict and then $\Delta_w <0$, too. Therefore, $S = \sum_{w \in V(G)}\Delta_w < 0$ and this contradicts the fact that $G$ is an equilibrium graph. 
\end{proof}


Next, we investigate further topological properties of any $2-$edge-connected component $H$ from any sum basic equilibrium $G$. These properties help us to derive the first main result of this paper.  

\begin{lemma}
\label{lem:bridge}
If $uv \in E(G)$ is a bridge, then $deg(u)=1$ or $deg(v)=1$.
\end{lemma}

\begin{proof}
Let $u_1u_2 \in E(G)$ be a bridge between two connected components $G_1,G_2$ in such a way that $u_1 \in V(G_1)$ and $u_2 \in V(G_2)$. Furthermore, assume wlog that $|V(G_1)| \leq |V(G_2)|$. If we suppose the contrary, then we can find a node $v\in V(G_1)$ such that $vu_1 \in E(G_1)$. Then, let $\Delta C$ be the cost difference associated to the deviation in $v$ that consists in swapping the edge $vu_1$ for the edge $vu_2$. Clearly, we are getting one unit closer to every node from $V(G_2)$ and getting one unit distance further from at most all nodes in $V(G_1)$ except for the node $v$ itself. Therefore, using the assumption $|V(G_1)| \leq |V(G_2)|$, we deduce: 
\begin{align*}
\Delta C &\leq |V(G_1)|-1-|V(G_2)| \leq -1 < 0.
\qedhere \end{align*}
\end{proof}

\begin{lemma}
\label{lem:pels}
If $H$ is any $2-$edge-connected component of $G$ then there exists at most one node $u \in V(H)$ such that $W_H(u) \neq \left\{ u \right\}$. 
\end{lemma}

\begin{proof}
Suppose the contrary and we reach a contradiction. Let $u_1,u_2$ be two distinct nodes such that $W_H(u_1) \neq  \left\{ u_1 \right\}$ and $W_H(u_2) \neq \left\{ u_2 \right\}$.  Let $v_1 \neq u_1$ and $v_2 \neq u_2$ be two nodes from $W_H(u_1)$ and $W_H(u_2)$ respectively.   By Lemma \ref{lem:bridge}, $W_H(u_1)$ and $W_H(u_2)$ are stars. Assume wlog that $D(u_1) \leq D(u_2)$. When swapping the link $v_2u_2$ for the link $v_2u_1$ we can reach the nodes from $V(G) \setminus \left\{ v_1\right\}$ at the distances seen by $v_1$ and, also, we are reducing in at least one unit distance the distance from $v_2$ to $v_1$. Therefore, if $\Delta C$ is the cost difference associated to such swap, then: $\Delta C \leq D(u_1)-D(u_2) - 1 < 0$.
\end{proof}

Therefore, combining  these two lemmas with Theorem \ref{thm:diamH}, we deduce that every non-tree bipartite sum basic equilibrium is the complete bipartite $K_{r,s}$ with some star $S_k$ (the star with a central node and $k$ edges) attached to exactly one of the nodes from $K_{r,s}$, let it be $x_0 \in V(K_{r,s})$. 
Then, if we consider any path $x_0-x_1-x_2$ in $H$ of length $2$, $x_2$ has an incentive to swap the link $x_2x_1$ for the link $x_2x_0$ unless $k=0$, that is, unless $G = K_{r,s}$. 


Now we are ready to state the main result of this section:
\begin{corollary}
The set of bipartite sum basic equilibria is the set of complete bipartite graphs $K_{r,s}$, with $r,s \geq 1$ and therefore the diameter of every bipartite sum basic equilibrium graph is at most $2$. 
\end{corollary}

\section{Non-bipartite graphs}
\label{sec:other}

In this section we consider some well-known families of non-bipartite graphs that have a tree-like topology, and we will see that if we require them to be sum basic equilibrium networks, then their diameter is at most 2 as it happens with the tree topology. 

First, we consider \emph{block graphs} (also known as \emph{clique trees}). A connected graph is a \emph{block graph} or \emph{clique tree} if its blocks ($2$-connected components) are cliques. Block graphs form a class that plays an important role in computer science, especially in the study of hooking networks \cite{DH2019} and sparse matrix algorithms \cite{BP1993}. 

In the next result we show that Theorem \ref{thm:basic}, which holds for trees, can be extended to the more general setting of block graphs. We should note that the new result is not a subcase of the previous bipartite case. Before we need the following preliminary result.

\begin{lemma}
\label{lemm:bridges-triangles}
Let $H$ be a biconnected component from $G$ with $u_1,u_2$ two cut vertices from $G$ with  $u_1u_2\in E(G)$. Furthermore, let $W_1 = W_H(u_1),W_2 =W_H(u_2)$. If $G$ is a sum basic equilibrium graph then $|W_1| = 1$ or $|W_2| = 1$. 
\end{lemma}

\begin{proof}
Suppose the contrary, then there exist nodes $v_1,v_2$ with $v_1u_1 \in E(W_1)$ and $v_2u_2 \in E(W_2)$. First of all, let us suppose that $v_1$ swaps the edge $v_1u_1$ for the edge $v_1u_2$ and let $\Delta C_1$ be such cost difference. Let $z$ be any node from $G$ and let $\Delta(z)$ be the difference of the distance from $v_1$ to $z$ in the deviated graph minus the distance from $v_1$ to $z$ in the original graph $G$. We consider the following cases:

\vskip 5pt

1. If $z \in W_1$.

\hspace{0.75cm} 1.a If there exists a shortest path from $G$ connecting $v_1$ with $z$ not using the edge $v_1u_1$, then $\Delta(z) = 0$. (Notice that $z = v_1$ belongs to this case).

\hspace{0.75cm} 1.b Otherwise, $\Delta(z) \leq 1$ because $u_1u_2 \in E(H)$ by hypothesis.

2. If $z \in W_2$ then $\Delta(z) = -1$.

3. If $z \in G \setminus \left(W_1 \cup W_2 \right)$ then $\Delta(z) \leq d_G(v_2,z)-d_G(v_1,z)$. 

\vskip 5pt
Therefore: 
$$\Delta C_1 \leq D_{G \setminus (W_1 \cup W_2)}(v_2)-D_{G \setminus (W_1 \cup W_2)}(v_1) +(|W_1|-1)-|W_2|.$$
Now, if we consider $\Delta C_2$ the cost difference associated to the deviation in $v_2$ that consists in swapping the edge $v_2u_2$ for the edge $v_2u_1$ we obtain: 
$$\Delta C_2\leq D_{G \setminus (W_1 \cup W_2)}(v_1)-D_{G \setminus (W_1 \cup W_2)}(v_2) +(|W_2|-1)-|W_1|$$
And from here: 
$$\Delta C_1 + \Delta C_2 \leq -2 < 0.$$
So $G$ cannot be a sum basic equilibrium graph. 
\end{proof}





\begin{corollary}
\label{corol:block}
If a sum equilibrium graph $G$ in the basic network-creation game is a block graph, then it has diameter at most $2$.
\end{corollary}

\begin{proof}
Suppose the contrary, that the diameter is at least $3$. Then there exist at least three non-trivial and edge-disjoint cliques $K,K_1,K_2$ and two distinct vertices $u_1,u_2$ such that $u_1 \in K \cap K_1$, $u_2 \in K \cap K_2$ and $u_1u_2 \in E(K)$. Let $H$ be the biconnected component from $G$ that $K$ defines. Then $K_1\subseteq W_H(u_1)$ and $K_2 \subseteq W_H(u_2)$ so that $|W_H(u_1)|, |W_H(u_2)| > 1$, contradicting Lemma \ref{lemm:bridges-triangles}. 
\end{proof}


A \emph{cactus graph} (sometimes called a \emph{cactus tree}) is a connected graph in which any two simple cycles have at most one vertex in common.

\begin{lemma}
\label{lemm:cycle}
Let $c$ be a cycle in a cactus sum basic equilibrium graph $G$ and let $H$ be the connected component that $c$ defines. Then the length of $c$ is at most $5$ and when it is $4$ or $5$, $|W_H(x)| = |W_H(y)|$ for any $x,y \in V(H)$.
\end{lemma}

\begin{proof}
Suppose that $u_0,u_1,\ldots,u_{l-1}$ are the consecutive vertices defining $c$ with $l$ being the length of the cycle and let $W_i = W_H(u_i)$  for each $i=0,\ldots,l-1$. When taking the subindices modulo $l$ and swapping the edge $u_iu_{i+1}$ for the edge $u_iu_{i+2}$, the value of the corresponding cost difference $\Delta C_i$ is at most $|W_{i+1}|-|W_{i+2}|-|W_{i+3}|$ if $l > 5$ or $|W_{i+1}|-|W_{i+2}|$ if $4 \leq l \leq 5$. Therefore, $\Delta C_0 +\cdots + \Delta C_{l-1} < 0$ if $l > 5$ and in this case $G$ cannot be a sum basic equilibrium. Whereas if $4 \leq l \leq 5$, then $\Delta C_0 +\cdots+\Delta C_{l-1} = 0$ and since in this case $\Delta C_i = |W_{i+1}|-|W_{i+2}|$, then we must have $|W_i| = |W_j|$ for every $0 \leq i \leq j \leq l-1$ as claimed.
\end{proof}

\begin{corollary}
\label{corol:cycle}
Any cactus graph $G$ that is a sum basic equilibrium contains at most one cycle of length strictly greater than $3$.
\end{corollary}
\begin{proof}
Suppose the contrary, let $c, c'$ be two cycles from $G$ of length greater than $3$ and let $H$, $H'$ be the corresponding two biconnected components that $c,c'$ define, respectively. 

On the one hand, there exist positive constants $W,W'$ such that $|W_H(v)| = W$ and $|W_{H'}(v')| = W'$ for any $v \in c$ and any $v' \in c'$. 

On the other hand, since $G$ is a cactus graph by hypothesis there exist nodes $u \in c$ and $u' \in c'$ such that $c' \subseteq W_{H}(u)$ and $c \subseteq W_{H'}(u')$. 

If $v \in c$ and $v \neq u$, $W_H(v) \subseteq W_{H'}(u')$ and then $W = |W_H(v)| < |W_H(u')| = W'$. Then, by symmetry, $W' < W$ and therefore we have reached a contradiction.  
\end{proof}

Now, we can state the last result of this paper:

\begin{theorem}
If a sum equilibrium graph $G$ in the basic network-creation game is a cactus graph, then it has diameter at most $2$.
\end{theorem}

\begin{proof} 
If $G$ does not contain any cycle of length strictly greater than $3$ then $G$ is a block graph and the result follows by Corollary \ref{corol:block}. Otherwise, by Corollary \ref{corol:cycle}, $G$ contains exactly one cycle of length $4$ or $5$, let it be $c$ and let $H$ be the corresponding biconnected component that $c$ defines. Now let $u_1,u_2 \in V(H)$ be two consecutive nodes from $c$ and let $W_1 = W_H(u_1)$ and $W_2=W_H(u_2)$. 

On the one hand, by Lemma \ref{lemm:cycle} it follows that $|W_1|=|W_2|$. On the other hand, by Lemma \ref{lemm:bridges-triangles}, we have that $|W_1| = 1$ or $|W_2|=1$. Thus $|W_1|=|W_2|=1$, and since this applies for any two consecutive nodes from the cycle defined by $H$, we conclude that $G = H$. Therefore, as claimed, the diameter of $G$ is at most $2$.  
\end{proof}




\section{Concluding remarks}
First of all, it is important to note that dealing with swap deviations can be extremely challenging. This is illustrated, for instance, in the proofs of a constant price of anarchy for the sum classical network creation game which appear in \cite{Demaineetal:07} and \cite{Alvarezetal3}. In those proofs, the deviations considered only involve buying links or deleting a subset of at least two links and buying a link back to some other node. Thus, our research provides a deeper understanding of these deviations, contributing to the advancement of the needed methods to analyse sum basic network creation games.

Our work shows that sum basic equilibrium graphs have diameter at most 2, regardless of whether the graph is bipartite, a block graph or a cactus graph. This extends previous results for trees by Alon et al. \cite{AlonDHL:10}. Can this result be generalised to other non-bipartite graph classes? This remains an open question, and we hope that our contribution will inspire further research on this problem. 


\subsection*{Acknowledgements}
Aida Abiad is partially supported by the Dutch Research Council through the grant VI.Vidi.213.085 and by the Research Foundation Flanders through the grant 1285921N. 
Arnau Messegu\'e is supported in part by grants Margarita Sala and 2021SGR-00434

A preliminary version of this paper appeared in the proceedings of the 2021 \emph{European Conference on Combinatorics, Graph Theory and Applications} (EuroComb 2021).


\end{document}